\documentclass[12pt,reqno]{amsart}

\usepackage[english]{babel}
\usepackage[utf8]{inputenc}
\usepackage{amsmath}
\usepackage{amssymb}
\usepackage{amsthm}
\usepackage{enumerate}
\usepackage{graphicx}
\usepackage{fullpage}
\usepackage{cite}
\newtheorem{thm}{Theorem}[section]
\newtheorem{lemma}[thm]{Lemma}
\newtheorem{cor}[thm]{Corollary}

\newtheorem{conjecture}[thm]{Conjecture}
\newtheorem{defn}[thm]{Definition}

\newcommand\contract{/}

\title{The sandpile group of a polygon flower}

\author{Haiyan Chen}
\address{%
The School of Sciences\\
Jimei University\\
Fujian, China}
\email{chey5@jmu.edu.cn}

\author{Bojan Mohar}
\address{%
Department of Mathematics\\
Simon Fraser University\\
Burnaby, BC, Canada}
\email{mohar@sfu.ca}

\thanks{This work was done while the first author visited The Simon Fraser University. The hospitality of the hosting institution is greatly acknowledged. The visit was funded by the Fujian Provincial Education Department.}
\thanks{H.Y.Chen was supported by the National Natural Science Foundation of China (grant numbers
11771181, 11571139).}
\thanks{B.M.~was supported in part by the NSERC Discovery Grant R611450 (Canada), by the Canada Research Chairs program, and by the Research Project J1-8130 of ARRS (Slovenia).}
\thanks{On leave from IMFM \& FMF, Department of Mathematics, University of Ljubljana.}%

\date{\today}

\begin{document}

\begin{abstract}
Let $C_t$ be a cycle of length $t$, and let $P_1,\ldots,P_t$ be $t$ polygon chains. A polygon flower $F=(C_t; P_1,\ldots,P_t)$ is a graph obtained by identifying the $i$th edge of $C_t$ with an edge $e_i$ that belongs to an end-polygon of $P_i$ for $i=1,\ldots,t$. In this paper, we first give an explicit formula for the sandpile group $S(F)$ of $F$, which shows that the structure of $S(F)$ only depends on the numbers of spanning trees of $P_i$ and $P_i\contract e_i$, $i=1,\ldots,t$.  By analyzing the arithmetic properties of those numbers, we give a simple formula for the minimum number of generators of $S(F)$, by which a sufficient and necessary condition for $S(F)$ being cyclic is obtained. Finally, we obtain a classification of edges that generate the sandpile group.

Although the main results concern only a class of outerplanar graphs, the proof methods used in the paper may be of much more general interest.  We make use of the graph structure to find a set of generators and a relation matrix $R$, which has the same form for any $F$ and has much smaller size than that of the (reduced) Laplacian matrix, which is the most popular relation matrix used to study the sandpile group of a graph.
\end{abstract}

\maketitle

\section{Introduction}

The abelian sandpile models were firstly introduced in 1987 by three physicists Bak, Tang,
and Wiesenfeld \cite{Bak1987Self}, who studied it mainly on the integer grid graphs. In 1990, Dhar\cite{Dhar1990Self} generalized their model from a grid to arbitrary graphs.
The abelian sandpile model of Dhar begins with a
connected graph $G = (V, E)$ and a distinguished vertex $q\in V $, called the \emph{sink}. A \emph{configuration} of $(G, q)$ is a vector $\vec{c}\in \mathbb{N}^{V-q}$. A non-sink vertex $v$ is \emph{stable} if its degree satisfies $d(v) > \vec{c}(v)$; otherwise it is \emph{unstable}. Moreover, a configuration is \emph{stable}
if every vertex $v$ in $V-q$ is stable. \emph{Toppling} an unstable vertex $u\in V-q$ in $\vec{c}$ is the operation performed by decreasing
$\vec{c}(u)$ by the degree $d(u)$, and for each neighbour $v$ of $u$ different from $q$, adding the multiplicity $m(u,v)$ of the edge $uv$ to $\vec{c}(v)$. Starting from any initial configuration $\vec{c}$, by performing a sequence of topplings, we eventually arrive at a stable configuration. It is not hard to see that the stabilization of an unstable configuration is
unique \cite{Dhar1990Self,Biggs1999}. The stable configuration associated to $\vec{c}$ will be denoted by $s(\vec{c})$.
Now, let $(\vec{c} + \vec{d})(u):= \vec{c}(u) + \vec{d}(u)$ for all $u\in V-q$ and $\vec{c} \oplus \vec{d }:= s(\vec{c} + \vec{d})$. A configuration $\vec{c}$ is \emph{recurrent}
if it is stable and there exists a non-zero configuration $\vec{r}$ such that $s(\vec{c}+\vec{r}) = \vec{c}$. Dhar\cite{Dhar1990Self} proved that the number of recurrent configurations is equal to the number of spanning trees of $G$, and that the set of recurrent configurations with $\oplus$ as a binary operation forms a finite abelian group, which is called the \emph{sandpile group} of $G$.
%, denoted by $S(G,q)$.
Soon after that, it was found that the sandpile group is isomorphic
to a number of `classical' abelian groups associated with graphs, such as the group of components in Arithmetic Geometry \cite{Lorenzini1989Arithmetical}, Jacobian group and Picard group in Algebraic Geometry\cite{Bak2007Riemann}, the
determinant group in lattice theory\cite{Bacher1997The},
the critical group of a dollar game \cite{Biggs1997Algebraic,Biggs1999}.

As an abstract abelian group, the structure of the sandpile group is independent of the choice of the sink $q$. We denote the sandpile group of $G$ by $S(G)$. The classification theorem for finite abelian groups asserts that $S(G)$ has a direct sum
decomposition $$S(G)=\mathbb{Z}_{d_1}\oplus\mathbb{Z}_{d_2}\oplus\cdots \oplus\mathbb{Z}_{d_r},$$
where the integers $d_1>1,d_2,\ldots, d_r$ are called \emph{invariant factors} of $S(G)$ or $G$, and they satisfy
$d_i \mid d_{i+1}$ for $i=1,\dots,r-1$.

The standard method of obtaining invariant factors of a finite abelian group is, first, choosing a presentation of the group, then computing the Smith normal form of the matrix of relations. The most popular presentation of the sandpile group is the Jacobian group model. One of the reasons is that under the presentation of the Jacobian, generators and relations can be chosen so that the relation matrix is the well-known reduced Laplacian matrix of the graph. Using this presentation, the sandpile groups for many special families of graphs have been completely or partially determined, see for instance \cite{Chen2008On,Chen2006On,Hou2006On,Christianson2002The,Deryagina2014On,Krepkiy2013The, Jacobson2003Critical,Musiker2009The,Shi2011The,Toumpakari2007On,
Glass2014Critical,Reiner2014Critical,Alfaro2012On,Chandler2015The,Berget2012The, Chan2014Sandpile,Ducey2018The,Ducey2017On,Raza2015On,Bai2003On,Glass2017Critical,
Bond2016The,Levine2011Sandpile} and references therein.

Given a graph $G$, directly choosing the reduced Laplacian matrix $\tilde{L}(G)$ as the relation matrix of $S(G)$ is an easy and convenient start, but to obtain the Smith normal form of $\tilde{L}(G)$ is not an easy task when the order of $G$ is large although there exists a polynomial algorithm \cite{Kannan1979}. Moreover, this method does not take into account the combinatorial structure of
the graph. In fact, for most of graphs, the minimum number of generators of $S(G)$, denoted by $\mu(G)$, is considerably smaller than the order of $G$.  For example, $\mu(T)=0$ for any tree $T$, and $\mu(G)=1$ for any uni-cyclic graph  $G$. In \cite{Lorenzini2008Smith}, Lorenzini asked about the problem-how often the sandpile group is cyclic, that is $\mu(G)=1$. Based on
a Cohen-Lenstra heuristic and empirical evidence, it has been conjectured in \cite{Clancy2015On} that the probability of cyclic
sandpile group in an Erd\H{o}s-R\'{e}nyi random graph tends to
$\prod_{i=1}^\infty \zeta(2i+1)^{-1}\approx 0.79$ as the
number of vertices goes to infinity. And, in \cite{Wood2017The}, Wood proved this to be an upper bound.

A natural idea when considering the sandpile group, is to use the structure of a graph to reduce the number of generators as much as possible, and then determine the Smith normal form of the smaller relation matrix. When this works, we not only simplify computation, but also gain additional insight into sandpile groups.  In this paper, we show that, for a large family of outerplanar graphs, the structure provides a set of generators, and the size of this set is dramatically smaller than the order of the graph.

The paper is organized as follows. In Section 2, we cover preliminaries. In Section 3, by using the structure properties of polygon chains, we not only show that the sandpile group of any polygon chain is cyclic, but determine the order of any ``edge" as an element of the group (see Theorem \ref{thm:3.2} and its Corollary \ref{cor:3.3}). These results are used in Section 4. For a polygon flower $F=(C_t; P_1,\ldots,P_t)$, we first describe a set of generators of $S(F)$ with $t$ elements and its corresponding relation matrix (Theorem \ref{thm:4.1}). By analyzing the relation matrix, we find an explicit formula for the sandpile group $S(F)$ (Theorem \ref{thm:4.3}). Based on the formula, the minimum number of generators of $S(F)$ is obtained (Theorem \ref{thm:4.6}). In Section 5, we discuss the generating edges when $S(F)$ is cyclic (Theorems \ref{thm:5.3} and \ref{thm:5.4}), and provide a general way to reduce the relation matrix to be as small as possible (Theorem \ref{thm:5.6}).

%%%%%%%%%%%%%%%%%%%%%%%%%%%%%%%%%%%%%%%%%%
\section{Preliminaries}

Let $G=(V,E)$ be a connected graph with $n$ vertices and $m$ edges. Given an arbitrary orientation $\mathcal{O}$ of $E$, and an oriented edge $e=(u,v)$, $v$ is called the \emph{head} of $e$, denoted by $h(e)$, and $u$ is called the \emph{\emph{tail}} of $e$, denoted by $t(e)$. As the convention, if $e=(u,v),$ then $ -e=(v,u)$.  Let $\mathbb{Z}V, \mathbb{Z}E$ denote the free abelian groups on $V$ and $E$, respectively. More clearly, every element $x\in \mathbb{Z}V $ is identified with formal sum $\sum_{v\in V(G)}x(v)v$, where $x(v)\in \mathbb{Z}$, and similarly for $y\in\mathbb{Z}E$.

Consider a cycle $C= v_1e_1v_2e_2\cdots v_{k}e_kv_1$ in the undirected
graph $G$. The sign of an edge $e$ in $C$ with respect to the orientation $\mathcal{O}$ is $\sigma(e;C) = 1$
if $C = vev$ is a loop at the vertex $v$, and otherwise
$$
\sigma(e;C)=\left\{
\begin{aligned}
1,&\ \ \ \mbox{if}\ \ e\in C \  \mbox{and}\ \ t(e)=v_i, h(e)=v_{i+1}\ \  \mbox {for some}\ \  i;\\
-1,&\ \ \ \mbox{if}\ \ e\in C\ \  \mbox{and}\ \  t(e)=v_{i+1}, h(e)=v_i\ \  \mbox {for some}\ \  i;\\
0,&\ \ \ \mbox{otherwise\ ($e$ does not occur in $C$)}.
\end{aligned}
\right.
$$
Here we interpret indices module $k$, i.e., $v_{k+1}=v_1$.
We then identify $C$ with the formal sum
$\sum_{e\in E}\sigma(e,C)e\in \mathbb{Z}E$.

For each nonempty $U\subset V $, the
cut corresponding to $U$, denoted by $c_U$, is the collection of edges with one end vertex
in $U$ and the other in the complement $\overline{U}$. For each $e\in E$, define the sign of $e$ in
$c_U$ with respect to the orientation $\mathcal{O}$ by

$$
\sigma(e;c_U)=\left\{
\begin{aligned}
1,\ \ \  &\mbox{if}\ \ t(e)\in U \ \ \mbox{and}\ \  h(e)\in \overline{U};\\
-1,\ \ \  &\mbox{if}\ \  t(e)\in \overline{U}\ \ \mbox{and}\ \  h(e)\in U;\\
0,\ \ \  &\mbox{otherwise\ ($e$ does not occur in $c_U$)}.
\end{aligned}
\right.
$$
We then identify $c_U$ with the formal sum $\sum_{e\in E}\sigma(e,c_U)e\in \mathbb{Z}E$.

A vertex cut is the cut corresponding to a single vertex, $U =\{v\}$,
and we write $c_v$ for $c_U$ in this case.

\begin{defn}
The (integral) cycle space, $\mathcal{C}\subseteq \mathbb{Z}E$, is the $\mathbb{Z}$-span of all cycles. The (integral) cut space, $\mathcal{B}\subseteq \mathbb{Z}E$, is the $\mathbb{Z}$-span of all cuts.
\end{defn}

Let $L(G)$ be the Laplacian matrix of $G$. It can be viewed as a (linear) mapping $\mathcal{L}:$ $\mathbb{Z}V\rightarrow\mathbb{Z}V$. We also define a mapping $\rho:$ $\mathbb{Z}V\rightarrow\mathbb{Z}$ as $\rho(\sum_{v\in V}x(v)v)=\sum_{v\in V}x(v).$

Obviously, both $\mathcal{L}$ and $\rho$ are group homomorphisms. Then
we have the following well-known results.

\begin{thm}
Let $G=(V,E)$ be a graph. With the notation defined above, we have
$$S(G)\cong \frac{Ker(\rho)}{Im(\mathcal{L})}\cong\frac{\mathbb{Z}E}{\mathcal{C}\oplus \mathcal{B}},$$
where $Ker(.)$ and $Im(.)$ denote the kernel and the image of a mapping.
\end{thm}

The middle presentation of the sandpile group in Theorem 2.2 is the well-known Jacobian group (also known as Picard group) of the graph. The Jacobian presentation has a natural set of generators for $S(F)$, for which the reduced Laplacian matrix $\tilde{L}(G)$ of $G$ is a relation matrix. For more details, see \cite{Biggs1999}.

Here we mainly focus on the second presentation. For any $e=(u,v)\in E$, let $\delta_e=\sum_{f\in E}\delta_e(f)f\in \mathbb{Z}E$, where $\delta_e(f)=1$ if $e=f$ and $0$ otherwise. Then the collection of all $\delta_{e}$ is
a natural set of generators of the sandpile group $S(G)$, and the relations are given by the elements in $\mathcal{C}\oplus \mathcal{B}$. So to find a relation matrix, we only need to find a basis of the cycle space $\mathcal{C}$ and a basis of the cut space $\mathcal{B}$, respectively.

Now we recall the definition and basic properties of the Smith normal form of an integer matrix.
Let $M, N$ be two $n\times n$ integer matrices. The two matrices are called \emph{equivalent} if there exist two  integer matrices $P$ and $Q$ with integer inverses (i.e., $|\det(P)|=|\det(Q)|=1$) such that $PMQ=N$. We have the following well-known results.

\begin{thm}\label{th:1}
{\rm (1)} Each integer matrix $M$ with rank $r$, is equivalent to a diagonal matrix $diag(d_1,\ldots,d_r, 0, \ldots, 0)$, where $d_i\,|\,d_{i+1}$, $i=1,\ldots,r-1$, and all these integers are positive. Furthermore, the $d_i$ are uniquely determined by
$$d_i=\frac{\Delta_i}{\Delta_{i-1}},\  i=1, \ldots, r,$$
where $\Delta_i$ (called $i$-th determinant divisor) equals the greatest common divisor of all $i\times i$ minors of the matrix $M$ and ${\Delta_{0}:=1}$.

{\rm (2)} Let $A$ be a finite abelian group with presentation $A=\{g_1,\ldots,g_n\ |\ \sum_{j=1}^n m_{ij}g_j=0, i=1,\ldots,n\} $. If $M=(m_{ij})$ is equivalent to the diagonal matrix $diag(d_1,\ldots,d_r, 0, \ldots, 0)$ then
$$A\cong\mathbb{Z}_{d_1}\oplus\cdots\oplus \mathbb{Z}_{d_r}.$$
\end{thm}

The diagonal matrix in Theorem 2.3 (1) is called the \emph{Smith normal form} of $M$, and the integers $d_i$ are called \emph{invariant factors} of $M$. The matrix $M$ related to a presentation of the abelian group $A$ in part (2) of Theorem 2.3 is called the \emph{relation matrix} of $A$.

From (1) of the above theorem, we see that equivalent matrices have the same invariant factors. And (2) says that the invariant factors of $A$ are just the non-trivial invariant factors (those that are $\geq 2$) of its arbitrary relation matrix. So, to determine the structure of a finite abelian group, it is sufficient to find a set of generators and a complete set of relations among them, then compute the Smith normal form of the corresponding relation matrix.  In this paper, we shall start from the natural set of generators $\delta_e$ $(e\in E(G))$ to study the sandpile groups of outerplanar graphs.

Let $(k_1,\ldots,k_n)$ be a sequence of integers with each $k_i\geq 2$. Define the graph $G_0$ to be a path of one edge, and for each $1\leq i\leq n$, define the graph $G_i$ by starting with graph $G_{i-1}$ and adding a path of $k_i-1$ edges between any two consecutive vertices of the path added at the previous step.
The resulting graph $G_n$
will consist of a stack of polygons with $k_1, \ldots , k_n$ sides. So we call $G_n$ a \emph{polygon chain}. For convenience, $G_0$ is called the \emph{trivial chain}. For any non-trivial polygon chain $G_n$ $(n\geq 1)$, the first polygon isomorphic to $C_{k_1}$, and the last polygon isomorphic to $C_{k_n}$, are called \emph{end-polygons}. Any edge in one of the end-polygons, which is not contained in another polygon, is called a \emph{free edge} of $G_n$.

Let $C_t$ be a polygon of length $t$, and let $e_1,\ldots, e_t$ be its edges in the cyclic order.  Let $P_1,\ldots,P_t$ be $t$ polygon chains. \emph{A polygon flower} $F=(C_t; P_1,\ldots,P_t)$ is any graph obtained by identifying a free edge $e_i'$ of $P_i$ with $e_i$ for $i=1,\ldots,t$. For a polygon flower $F = F(C_t; P_1,\ldots,P_t)$, the central cycle $C_t$ is called the \emph{flower center} and any non-trivial polygon chain $P_i$ is called a \emph{petal}.  From the definition, it is easy to see that $F=(C_t; P_1,\ldots,P_t)$ is an outerplanar graph, and a polygon chain can be viewed as a polygon flower with less than $3$ petals.

Finally, recall that for any connected graph $G$ and any $u\in V(G)$, the cuts $c_v, v\in V(G)-u$ form a basis of the cut space. While for any plane graph, the cycles corresponding to the bounded faces form a basis of the cycle space. Since the polygon flower $F=(C_t; P_1,\ldots,P_t)$ is an outerplanar graph,  the polygons (cycles) in $F$ form a basis of the cycle space of $F$.  In the following, for simplicity, we write $e$ for $\delta_e $. Given $t$ integers $a_1,\ldots,a_t$, we write $gcd(a_1,\ldots,a_t)$ for their greatest common divisor.

\section{The sandpile group of a polygon chain}

In this section, we shall discuss the sandpile group of a polygon chain. Given a  polygon chain $G_n(k_1,\ldots,k_n)$ $(k_i\geq 2)$, let $e_0$ denote the edge in $G_0$, and let
$e_i$ denote the edge shared by the polygons $C_{k_i}$ and $C_{k_{i+1}},\  i=1,\ldots,n-1$. We also fix an arbitrary free edge in $C_{k_n}$ as $e_n$. First we give a result about the order of the sandpile group $S(G_n)$, which is equal to $\tau(G_n)$ (the number of spanning trees of $G_n$). The well-known deletion-contraction formula for the number of spanning trees of a graph $G$ gives
$$\tau(G)=\tau(G-e)+\tau(G\contract e),$$
where $G-e$ and $G\contract e$ denote the graphs obtained from $G$ by deleting and contracting the edge $e\in E(G)$, respectively. By using this formula, it is not difficult to derive the following recurrence.

\begin{figure}[htbp]
\centering
\scalebox{1.15}{\includegraphics{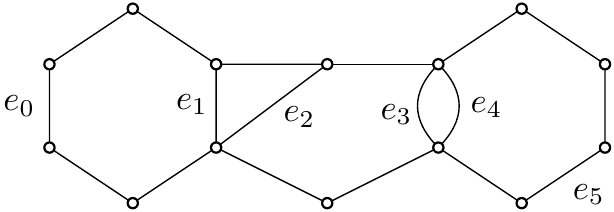}}
\caption{ Polygon chain $G_5(6,3,5,2,6)$ and its edges $e_0, \dots, e_5$.}
\end{figure}

\begin{lemma}
Given a polygon chain in $G_n(k_1,\ldots,k_n)$ $(k_i\geq 2)$, let $e_0,\ldots,e_n$ be the edges as defined above. Then
$$\tau(G_0)=1,\ \  \tau(G_1)=k_1;\ \  \tau(G_i)=\tau(G_{i-1})+\tau(G_{i}\contract e_{i}), \  1<i\leq n \eqno (3.1)$$
and
$$\tau(G_i)=(k_i-1)\tau(G_{i-1})+\tau(G_{i-1}\contract e_{i-1})=k_i\tau(G_{i-1})-\tau(G_{i-2}),\  1<i\leq n. \eqno (3.2)$$
\end{lemma}

Note that the polygon chains $G_n\ (n\geq 2) $ are not determined by the ordered array of cycle length $(k_1, \ldots, k_n)$ up to graph isomorphism. But by Lemma 3.1, it is easy to see that $\tau(G_n)$ only depends on $(k_1, \ldots, k_n)$, and it is independent of the way that the polygons stack together. Furthermore, it is known that the sandpile group of any polygon chain is cyclic \cite{Becker2016Cyclic,Krepkiy2013The}, so the sandpile group $S(G_n(k_1, \ldots, k_n)$ only depends on $(k_1, \ldots, k_n)$. This property can also be deduced from the fact that the sandpile groups of a planar graph and its dual are isomorphic \cite{Cori2000On,Berman1986}, since the dual graphs of $G_n(k_1,\ldots,k_n)$ are isomorphic.

In the following, we not only give another proof that $S(G_n)$ is cyclic, but give the information on the order of any element in $S(G_n)$. The last point is important for us to study the sandpile group of a general polygon flower.

Now we fix an orientation $\mathcal{O}$ of $E(G_n)$ as follows: first give an orientation of $e_n=(u_n,v_n)$; then the remaining edges in $C_{k_n}$ are oriented such that they form a directed path from $u_n$ to $v_n$. This determine the orientation of $e_{n-1}$. Suppose $e_{n-1}=(u_{n-1},v_{n-1})$, we orient the remaining edges in $C_{k_{n-1}}$ such that they form a directed path from $u_{n-1}$ to $v_{n-1}$. And so on, until all edges in $E(G_n)$ are oriented (see Figure 2 (a) for an example). Now we are ready to give the main result in this section.

\begin{figure}[htbp]
\centering
\scalebox{1.2}{\includegraphics{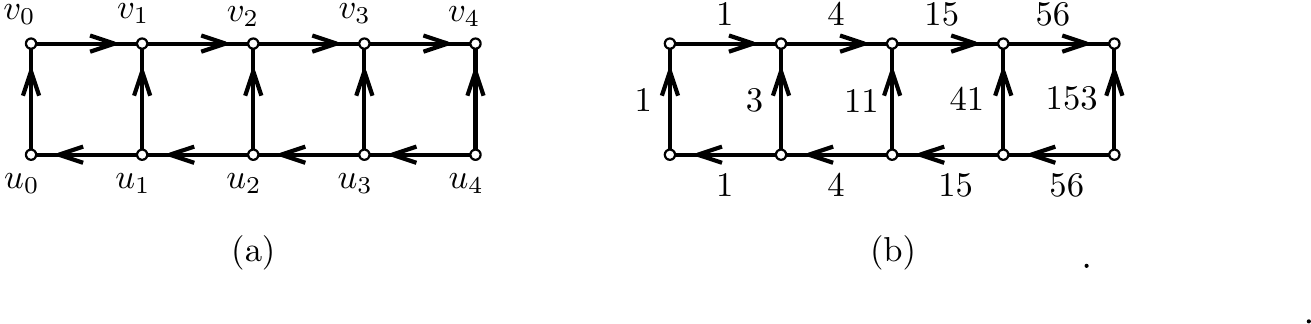}}
\caption{(a) The oriented polygon chain $G_4(4,4,4,4)$.\ (b) The coefficients of the edges expressed by $e_0$.}
\end{figure}

\begin{thm}\label{thm:3.2}
Let $G_n(k_1,\ldots,k_n)$ be a polygon chain with $e_0,\ldots,e_n$ given as above. Then, for any $e\in E(G)$ (viewed as an element in $S(G_n)=\frac{\mathbb{Z}E}{\mathcal{C\oplus \mathcal{B}}}$), we have
$$
e\equiv\left\{
\begin{aligned}
\tau(G_i\contract e_i){e_0},\ \ \  &\mbox{if}\ \  e=e_i,\  i=0,\ldots,n\\
\tau(G_{i-1}){e_0},\ \ \  &\mbox{if}\ \  e \in E(C_{k_i})\setminus \{e_{i-1},e_i\}, \ i=1,\ldots,n
\end{aligned}
\right.
$$
where $``\equiv"$ means mod $(\mathcal{C}\oplus\mathcal{B})$.
\end{thm}

\begin{figure}[htbp]
  \centering
  \scalebox{0.66}{\includegraphics{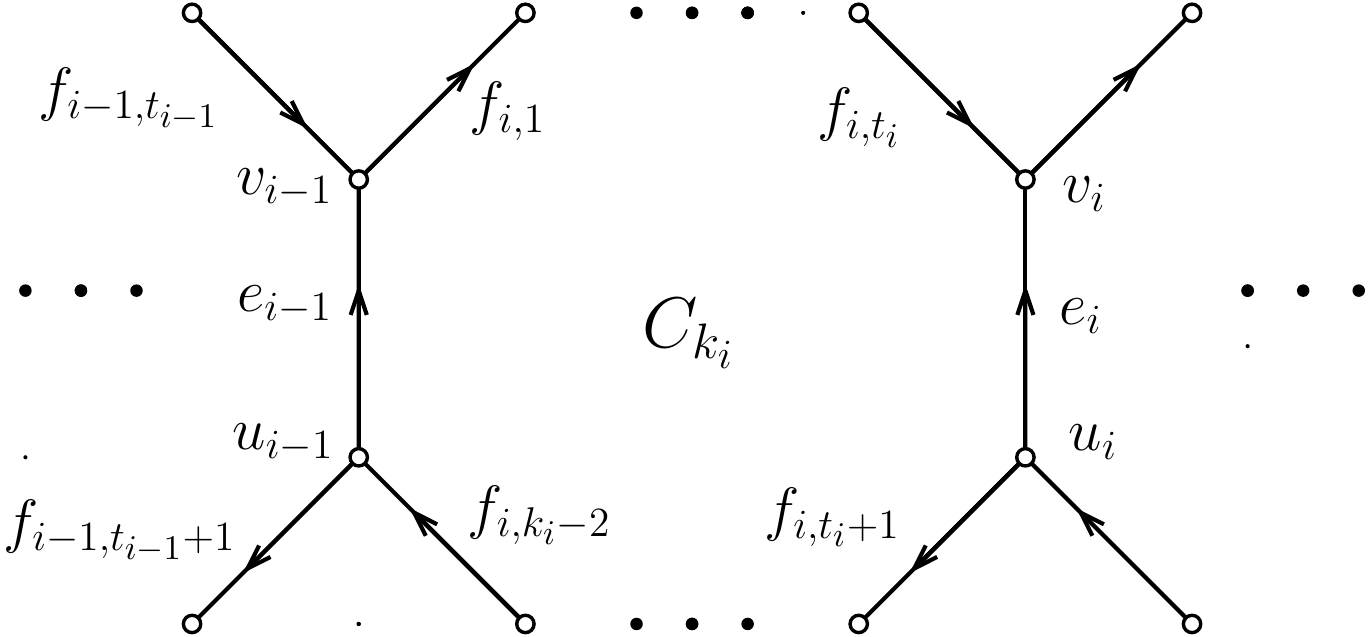}}
  \caption{A general oriented polygon chain.}
\end{figure}

\begin{proof} By using the relations determined by cuts $c_v, v\in V(G_n)$ and polygons $C_{k_i}$ alternatively, we shall show every edge $e\in E(G_n)$ can be expressed by $e_0$, and the coefficients are given as stated in the theorem. First note that the following simple fact, which we shall use repeatedly without pointing out every time:

if $e$ and $f$ are two edges incident with a vertex $v$ of degree-$2$ and $h(e)=t(f)=v$, then
$$e\equiv f\  (mod\  c_v).$$

Also note that, under the given orientation $\mathcal{O}$, the cycle $C_{k_i}$ (see Figure 3) can be expressed as follows:
$$C_{k_i}=e_{i-1}+f_{i,1}+\cdots +f_{i,t_i}-e_i+f_{i,t_i+1}+\cdots +f_{i,k_i-2},\qquad i=1,\cdots,n.$$

Now we start from the edges in $C_{k_1}$. If $k_1=2$, then $C_{k_1}=e_0-e_1$, so $e_1\equiv e_0\ (mod\  C_{k_1})=\tau(G_1\contract e_1)e_0$, and we are done. If $k_1>2$, then we have, for every $i$, $f_{1,i}\equiv e_0\  (mod(\mathcal{B}))=\tau(G_0)e_0$ since every edge $f_{1,i}$ is incident with a degree-2 vertex. Finally, we have $$e_1\equiv(k_1-1)e_0\  (mod\  C_{k_1})=\tau(G_1\contract e_1)e_0.$$

Now the proof proceeds by induction on $i$. Suppose that we have expressed the edges in $C_{k_{i-1}}$ as claimed. If $1\leq t_{i-1}<k_{i-1}-2$ and $1\leq t_{i}<k_{i}-2$, then
$$c_{v_{i-1}}=f_{i,1}-e_{i-1}-f_{i-1,t_{i-1}}\ \mbox{and}\   c_{u_{i-1}}=e_{i-1}+f_{i-1,t_{i-1}+1}-f_{i,k_i-2}.\eqno (3.3)$$
By the induction, $f_{i-1,t_{i-1}}\equiv f_{i-1,t_{i-1}+1}$. Thus, by summing up the cuts in (3.3), we obtain
$$f_{i,1}\equiv f_{i,k_i-2}\equiv e_{i-1}+f_{i-1,t_{i-1}}\equiv (\tau(G_{i-1}\contract e_{i-1})+\tau(G_{i-2}))e_0=\tau(G_{i-1})e_0.$$
Then, by using the cuts induced by the degree-2 vertices, we deduce $$f_{i,j}\equiv \tau(G_{i-1})e_0,\  1\leq j\leq k_i-2.$$
Finally, for $e_i$, we have
\begin{align*}
  e_i &\equiv e_{i-1}+f_{i,1}+\ldots +f_{i,k_i-2}\ (mod\ C_{k_i})\\
      &\equiv ((k_i-2)\tau(G_{i-1})+\tau(G_{i-1}\contract e_{i-1}))e_0 = \tau(G_i\contract e_i)e_0
\end{align*}
where the last equality is obtained by using (3.2). The proof is basically the same when at least one of conditions $t_{i-1}=0$ or $t_i=0$ or $t_{i-1}=k_{i-1}-2$ or $t_i=k_i-2$ holds. The details are left to the reader. So by the principle of induction, we have the result.
\end{proof}

Since every edge in $E(G_n)$ can be expressed by $e_0$, $S(G_n)$ is cyclic. Not only that, the above theorem gives additional information about the sandpile group $S(G_n)$. For example, the order of any element $a\in  S(G_n)$, denoted by $ord(a)$, is determined. By this, we can easily judge whether an element is a generator of $S(G_n)$ or not, in particular for an edge $e\in E(G_n)$.

\begin{cor}\label{cor:3.3}
Let $G_n(k_1,\ldots,k_n)$ be a polygon chain with edges $e_0,\ldots,e_n$ as defined above. Then we have the following.

{\rm (1)}\ For any $e\in E(G_n)$,
$$
ord(e)=\left\{
\begin{aligned}
\frac{\tau(G_n)}{gcd(\tau(G_i\contract e_i), \tau(G_n))},\ \ \  &\mbox{if}\ \  e=e_i,\  0\leq i\leq n;\\
\frac {\tau(G_n)}{gcd(\tau(G_{i-1}),\tau(G_n))},\ \ \  &\mbox{if}\ \  e \in E(C_{k_i})\setminus \{e_{i-1},e_i\},\ 1\leq i\leq n.
\end{aligned}
\right.
$$

{\rm (2)}\ For each $i\in \{0,\ldots,n\}$, $e_i$ is a generator if and only if $gcd(\tau(G_i\contract e_i), \tau(G_n))=1$. Similarly, $e \in E(C_{k_i})\setminus \{e_{i-1},e_i\}$ is a generator if and only if $gcd(\tau(G_{i-1}), \tau(G_n))=1$.
\end{cor}

From Theorem 3.2, it is easy to see that any free edge $e\in E(G_n)$ is a generator. So by Corollary 3.3 (2)
$$ gcd(\tau(G_n\contract e), \tau(G_n))=1. \eqno (3.4)$$
This fact is used repeatedly in our proofs in the following sections. In fact, (3.4) can also be proved by using Lemma 3.1 directly. First by (3.2), we deduce that
$$gcd(\tau(G_{n-1}), \tau(G_n))=\cdots=gcd (\tau(G_{i-1}), \tau(G_i))=\cdots=gcd(\tau(G_0), \tau(G_1))=1.$$
Then (3.4) follows from (3.1).

Besides the free edges, there may exist other generating edges for the sandpile group of $G_n$. For an example, see Figure 2 (b), where the number near an edge is the coefficient of the edge expressed by $e_0$. Since the number of spanning trees of this polygon chain is $209$, every edge, except $e_2$, is a generator.

\medskip

{\bf Remark:} In the expressing process of Theorem 3.2, the edge $e_i$ is the last edge expressed in $C_{k_i}$ by using the relation determined by the polygon $C_{k_i}$ itself. So when we express $e_n=(u_n,v_n)$, we leave two relations determined by the cuts $c_{u_n}, c_{v_n}$ unused. We should bear in mind this point.

\section{The sandpile group of a polygon flower}

Let $C_t=v_1e_1v_2e_2\ldots v_te_tv_1$ be a cycle of length $t$, and $P_1,\ldots,P_t$ be $t$ polygon chains. Recall that a polygon flower $F=F(C_t;P_1,\ldots,P_t)$ is constructed by identifying a free edge $e_i'$ of $P_i$ with $e_i$ for $i=1,\ldots,t$. We fix an orientation $\mathcal{O}$ of $E(F)$ as follows: first let $e_i=(v_i,v_{i+1}), i=1,\ldots,t$, then determine the orientation of the edges in each $P_i$ consistently with the oriented $e_i$ according to the rules given in Section 3 (see Figure 4 for two examples). Then we have the following result.

\begin{figure}[htbp]
  \centering
  \scalebox{0.75}{\includegraphics{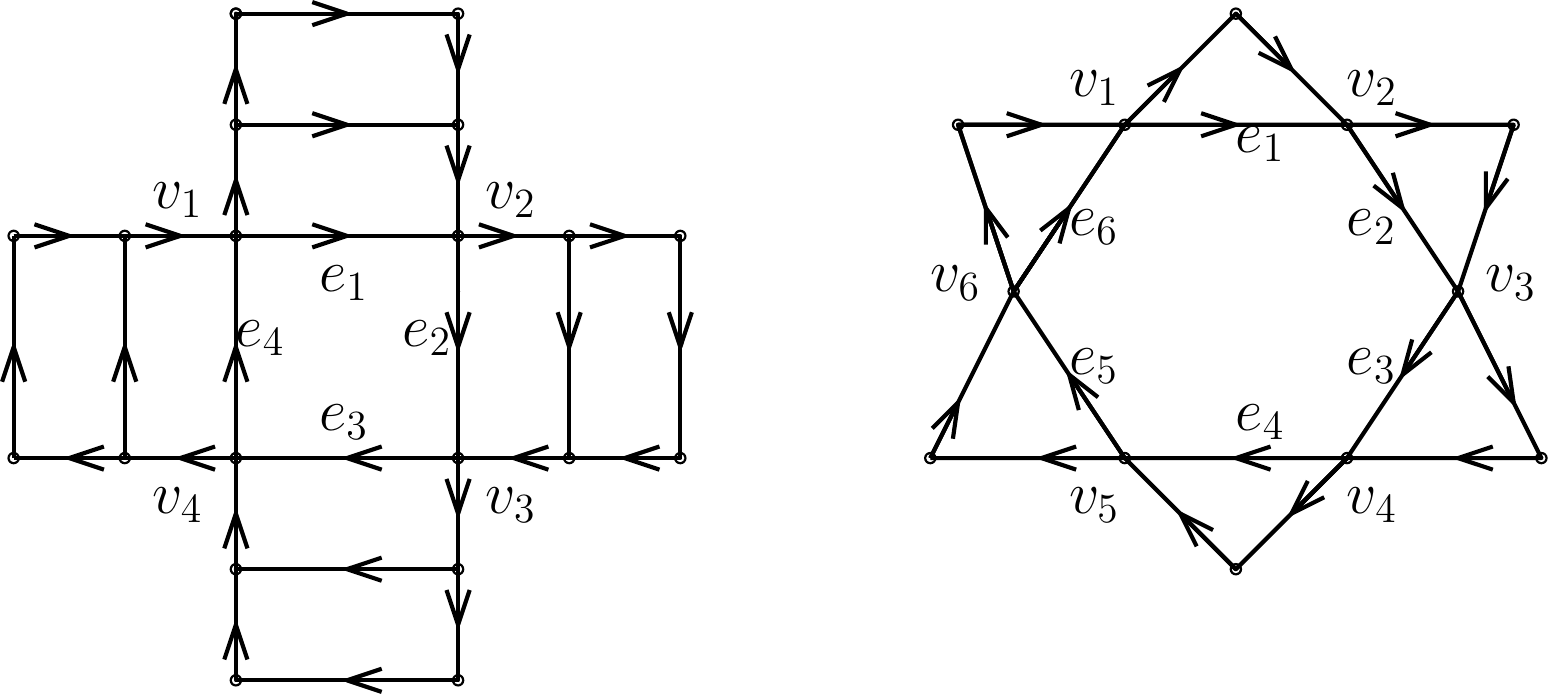}}
  \caption{Two oriented polygon flowers.}
\end{figure}

\begin{thm}\label{thm:4.1}
Let $F=F(C_t;P_1,\ldots,P_t)$ be a polygon flower with center $C_t=v_1e_1v_2e_2\ldots v_te_tv_1$. Then
there exist edges $f_i\in E(P_i), i=1,\ldots,t$, which generate the sandpile group $S(F)$. Furthermore, a relation matrix for these generators is as follows:
$$R=\left(
\begin{array}{cccccc}
  \tau(P_1)&-\tau(P_2)&0&\cdots&0&0\\
  0&\tau(P_2)&-\tau(P_3)&\cdots&0&0\\
  \cdots&\cdots&\cdots&\ddots&\cdots&\cdots\\
  0&0&0&\cdots&\tau(P_{t-1})&-\tau(P_t)\\
  \tau(P_1\contract e_1)&\tau(P_2\contract e_2)&\tau(P_3\contract e_3)&\cdots&\tau(P_{t-1}\contract e_{t-1})&\tau(P_t\contract e_t)
 \end{array}
\right).$$
\end{thm}

\begin{proof}
We choose an edge $f_i$ from $P_i$ for $i=1,\ldots,t$ as follows: if $P_i$ is a trivial chain, then choose $f_i=e_i$; if $P_i$ is a non-trivial chain with only one polygon $C$, then choose $f_i\in E(C)\setminus\{e_i\}$ arbitrarily;  if $P_i$ is a non-trivial chain with more than one polygon, then choose any free edge from the other end polygon in $P_i$ as $f_i$. As seen in the proof of Theorem 3.2, every edge in $P_i$ (viewed as an element in $S(P_i)$) can be expressed by $f_i$. In particular, we have $e_i\equiv\tau(P_i\contract e_i)f_i$, and $$\sum_{e\in P_i, h(e)=v_{i+1}}e=\sum_{e\in P_i, t(e)=v_{i}}e=\tau(P_i)f_i, \qquad i=1,\ldots,t.$$
So, the set $\{f_1, \ldots, f_t\}$ generates $S(F)$. For the relations among $f_i$, let us recall the remark we made at the end of the previous section, that the above expression process only uses the relations determined by the cuts $c_v$, $v\in V(F)\setminus\{v_1, \ldots, v_n\}$ and the cycles in $P_1,\dots,P_t$. So there remain $t-1$ independent relations determined by cuts
$$c_{v_i}=\sum_{t(e)=v_i,e\in P_i}e-\sum_{h(e)=v_i, e\in P_{i-1}}e, \qquad i=2,\ldots,t$$
and one additional cycle relation due to the center polygon $C_t$. That is,
$$
 \begin{aligned}
 &\tau(P_{1})f_1-\tau(P_{2})f_2\equiv 0,\\
 &\tau(P_{2})f_2-\tau(P_{3})f_3\equiv 0,\\
 &\ \ \ \ \ \ \ \ \ \ \vdots\ \ \ \ \ \ \ \ \ \ \\
 &\tau(P_{t-1})f_{t-1}-\tau(P_{t})f_t\equiv 0,\\
 &\tau(P_{1}\contract e_{1})f_1+\tau(P_{2}\contract e_{2})f_2+\cdots+\tau(P_{t-1}\contract e_{t-1})f_{t-1}+\tau(P_{t}\contract e_t)f_t\equiv 0.
\end{aligned}
$$

The corresponding relation matrix is $R$ given in the theorem.
\end{proof}

From the above theorem, we immediately derive the following corollary.

\begin{cor}
Let $F=F(C_t;P_1,\ldots,P_t)$ be a polygon flower. Then
$$ |S(F)|=\tau(F)=(\prod_{j=1}^t\tau(P_j))\sum_{i=1}^t\frac{\tau(P_i\contract e_i)}{\tau(P_i)}\eqno (4.1)$$
and for any permutation $\pi$ of the set $\{1,\ldots,t\}$, $$S(F(C_t;P_1,\ldots,P_t))\cong S(F(C;P_{\pi_1},\ldots,P_{\pi_t})).$$
\end{cor}

\begin{proof}
The formula (4.1) follows from that $|S(F)|=\tau(F)=|\det(R)|$ and the last statement is clear from the fact that the relation matrices $R$ and $\pi(R)$ are equivalent.
\end{proof}

Now we shall give a general formula for $S(F)$ by studying the relation matrix $R$.

\begin{thm}\label{thm:4.3}
Let $F=F(C_t;P_1,\ldots,P_t)$ be a polygon flower, and let $\mu(F)$ denote the minimum number of generators of $S(F)$. Let $d_0=1$ and for $k=1,\ldots,t-2$, $d_k=gcd(\tau(P_{i_1})\cdots\tau(P_{i_k})\ |\ 1\leq i_1<\cdots<i_k\leq t)$. Then
$$ S(F)=\mathbb{Z}_{\frac{d_1}{d_0}}\oplus \mathbb{Z}_{\frac{d_2}{d_1}}\oplus\cdots\oplus \mathbb{Z}_{\frac{d_{t-2}}{d_{t-3}}}\oplus \mathbb{Z}_{\frac{\tau(F)}{d_{t-2}}}$$
and
$$\mu(F)=t-1-k_0, $$
where $k_0=\max\{i\ |\ d_i=1, \ i=0,\ldots,t-2\}$.
\end{thm}

\begin{proof}
Since $R$ is a relation matrix of $S(F)$, by Theorem 2.3 (2), we have
$$ S(F)=\mathbb{Z}_{\Delta_1}\oplus \mathbb{Z}_{\frac{\Delta_2}{\Delta_1}}\oplus\cdots\oplus \mathbb{Z}_{\frac{\Delta_{t-1}}{\Delta_{t-2}}}\oplus \mathbb{Z}_{\frac{\tau(F)}{\Delta_{t-1}}},$$
where $\Delta_i$ is the $i$-th determinant factor of $R$.
So we only need to show that $$\Delta_k=d_{k-1},\ \ k=1,\ldots, t-1.$$
First, by (3.4), $gcd(\tau(P_1),\tau(P_1\contract e_1))=1$, so $\Delta_1=1=d_0$.

For $k=2,\ldots,t-1$,
on the one hand, it is easy to see that $d_{k-1}\,|\,\Delta_k$ since each $k\times k$ non-zero minor of $R$ is a linear combination of $\tau(P_{i_1})\cdots\tau(P_{i_{k-1}})$, $1\leq i_1<\cdots<i_{k-1}\leq t$.

On the other hand, we shall show that $\Delta_k\,|\,d_{k-1}$, that is, $\Delta_k$ divides every product
$$\tau(P_{i_1})\cdots\tau(P_{i_{k-1}}),\ \  1\leq i_1<\cdots<i_{k-1}\leq t.$$
Since $gcd(\tau(P_1),\tau(P_1\contract e_1))=1$, there exist integers $\alpha,\beta$ such that
$\alpha \tau(P_1)+\beta \tau(P_1\contract e_1)=1$. Let
$$M=\left(
\begin{array}{ccc}
  \alpha &0 &\beta\\
  0&I_{t-2} & 0\\
  -\tau(P_1\contract e_1) &0 &\tau(P_1)
\end{array}
\right),$$
where $I_{t-2}$ is the identity matrix of order $t-2$. Clearly, $M$ is invertible since $|\det(M)|=1$, so
$MR$ is equal to the matrix
$$\left(
\begin{array}{cccccc}
   1&\beta\tau(P_2\contract e_2)-\alpha\tau(P_2)&\beta\tau(P_3\contract e_3)&\cdots&\beta\tau(P_{t-1}\contract e_{t-1})&\beta\tau(P_t\contract e_t)\\
  0&\tau(P_2)&-\tau(P_3)&\cdots&0&0\\
  \cdots&\cdots&\cdots&\ddots&\cdots&\cdots\\[1.5mm]
  0&0&0&\cdots&\tau(P_{t-1})&-\tau(P_t)\\
  0&\tau(P_1)\tau(P_2\contract e_2)+\tau(P_2)\tau(P_1\contract e_1)&\tau(P_1)\tau(P_3\contract e_3)&\cdots&\tau(P_1)\tau(P_{t-1}\contract e_{t-1})&\tau(P_1)\tau(P_t\contract e_t)
 \end{array}
\right).$$
Apparently, each item of form $\tau(P_{i_1})\cdots\tau(P_{i_{k-1}}), 2\leq i_1<\cdots<i_{k-1}\leq t$ is a $k\times k$-minor of $MR$, so it is divisible by $\Delta_k$. Then by symmetry, we conclude that $\Delta_k$ divides each number $\tau(P_{i_1})\cdots\tau(P_{i_{k-1}}), 1\leq i_1<\cdots<i_{k-1}\leq t$.
Thus $\Delta_k\,|\,d_{k-1}$. So, $\Delta_k=d_{k-1},\quad  k=1,\ldots,t-1$, and this completes the proof of the first part. The second part follows directly from this.
\end{proof}

Note that for a trivial chain $P=e$, $\tau(P)=\tau(P\contract e)=1$. So the above results can be expressed as follows.

\begin{thm}
Let $F=F(C_t;P_1,\ldots,P_t)$ be a polygon flower with $s$ petals, say $P_1,\ldots,P_s$, and let $\mu(F)$ denote the minimum number of generators of $S(F)$. Then
$$|S(F)|=\tau(F)=\prod_{i=1}^s\tau(P_i)(t-s+\sum_{i=1}^{s}\frac{\tau(P_i\contract e_i)}{\tau(P_i)})\eqno (4.2)$$ and
$$ S(F)=\mathbb{Z}_{\frac{d_1}{d_0}}\oplus \mathbb{Z}_{\frac{d_2}{d_1}}\oplus\cdots\oplus \mathbb{Z}_{\frac{d_{s-2}}{d_{s-3}}}\oplus \mathbb{Z}_{\frac{\tau(F)}{d_{s-2}}}\eqno(4.3)$$
where $d_k=gcd(\tau(P_{i_1})\cdots\tau(P_{i_k})\ |\ 1\leq i_1<\cdots<i_k\leq s)$ for $k=1,\ldots,s-2$ and $d_0=1$.
Moreover
$$\mu(F)=s-1-k_0\eqno (4.4) $$
where $k_0=\max\{i\ |\ d_i=1, \ i=0,\ldots,s-2\}$.
\end{thm}

Theorem 4.4 says that, for a polygon flower $F$ with $s$ petals, the minimum number of generators $\mu(G)$ can be any number between $1$ and $s-1$. Furthermore,
$\mu(F)= s-1$ if and only if $gcd(\tau(P_1),\cdots,\tau(P_s))> 1$; while $\mu(F)=1$ if and only if $gcd(\tau(P_{i_1})\ldots\tau(P_{i_{s-2}})\ |\ 1\leq i_1<\cdots<i_{s-2}\leq s)=1$. In particular, if  $s=1$ or $2$, $S(F)$ is cyclic, which is consistent with the result in Section 3 for polygon chains. In fact, from Theorem 4.3 or 4.4, we can deduce a series of exact results. First we give an exact result for the extreme case when $\mu(F)=s-1$.

\begin{cor}
Let $F=F(C_t;P_1,\ldots,P_t)$ be a polygon flower with $s$ petals, say $P_1,\ldots,P_s$.
If $\tau(P_1)=\tau(P_2)=\cdots=\tau(P_s)=a$, then
$$ S(F)=\mathbb{Z}_a^{s-2}\oplus \mathbb{Z}_{ra},$$
where $r=(t-s)a+\sum_{i=1}^s\tau(P_i\contract e_i).$
\end{cor}

Now, given a positive integral vector  $\vec{a}=(a_1, \ldots, a_t)$ with at least one $a_i>1$, we define $m(\vec{a})$ be the maximum of numbers $l$ such that there exist $l$ integers $a_{i_1}, \ldots, a_{i_l}$ with $gcd(a_{i_1},\ldots,a_{i_l})> 1 $. For example, if $\vec{a}=(2, 2, 3,3,5,5)$ and $\vec{b}=(6, 10, 15,105)$, then $m(\vec{a})=2$ and $m(\vec{b})=3$. Then we have the following result.

\begin{thm}\label{thm:4.6}
Let $F=F(C_t;P_1,\ldots,P_t)$ be a polygon flower, and let $\vec{p}=(\tau(P_1),\ldots,\tau(P_t))$. Then the minimum number of generators of $S(F)$ is
$$
\mu(F)=\left\{
\begin{aligned}
1,\ \ \  &\mbox{if}\ \  m(\vec{p})=1;\\
m(\vec{p})-1,\ \ \  &\mbox{if}\ \  m(\vec{p})>1.
\end{aligned}
\right.
$$
\end{thm}

\begin{proof}
 Recall that $\mu(F)=t-1-k_0$, where $k_0=\max\{i\ | \ d_i=1,\ i=0,\ldots,t-2\}$. So we only need to show that $k_0=t-m(\vec{p})$. First,
 without loss of generality, suppose that $$gcd(\tau(P_1),\ldots,\tau(P_{m(\vec{p})}))=d>1.$$ Note that $d_{t-m(\vec{p})+1}=gcd(\tau(P_{i_1})\cdots\tau(P_{i_{t-m(\vec{p})+1}})\ |\ 1\leq i_1<\cdots<i_{t-m(\vec{p})+1}\leq t)$. Since each term $\tau(P_{i_1})\cdots\tau(P_{i_{t-m(\vec{p})+1}})$ includes at least one $\tau(P_i)$ with $1\leq i\leq m(\vec{p})$, we have $d\,|\,d_{t-m(\vec{p})+1}$. So $d_{t-m(\vec{p})+1}> 1$, thus $k_0\leq t-m(\vec{p})$. On the other hand, for any prime $q$, by the definition of $m(\vec{p})$, $q$ divides at most $m(\vec{p})$ terms of $\tau(P_1),\ldots,\tau(P_t)$. Any product of the remaining $t-m(\vec{p})$ terms is not divisible by $q$. Thus
 $q\nmid d_{t-m(\vec{p})}$, which implies that $d_{t-m(\vec{p})}=1$. So $k_0\geq t-m(\vec{p})$. This completes the proof.
\end{proof}

From the above theorem, we immediately derive the following result.

\begin{cor}
Let $F=F(C_t;P_1,\ldots,P_t)$ be a polygon flower, and let $\vec{p}=(\tau(P_1),\ldots,\tau(P_t))$. Then $S(F)$ is cyclic if and only if $m(\vec{p})\leq 2$.
\end{cor}

At the end of this section, we use above results to determine sandpile groups of some special graphs.

\subsection*{ Example 1. Outerplanar graphs with at most $8$ vertices}

First note that, except $G_{23}$ in Figure 5, every 2-connected outerplanar graph with at most $8$ vertices is a polygon flower. Since the sandpile group of any polygon flower with less than $3$ petals is cyclic, here we only list the graphs with at least $3$ petals (the center is labelled by $C$ in Figure 5).

\begin{figure}[htbp]
 \centering
 \scalebox{0.42}{\includegraphics{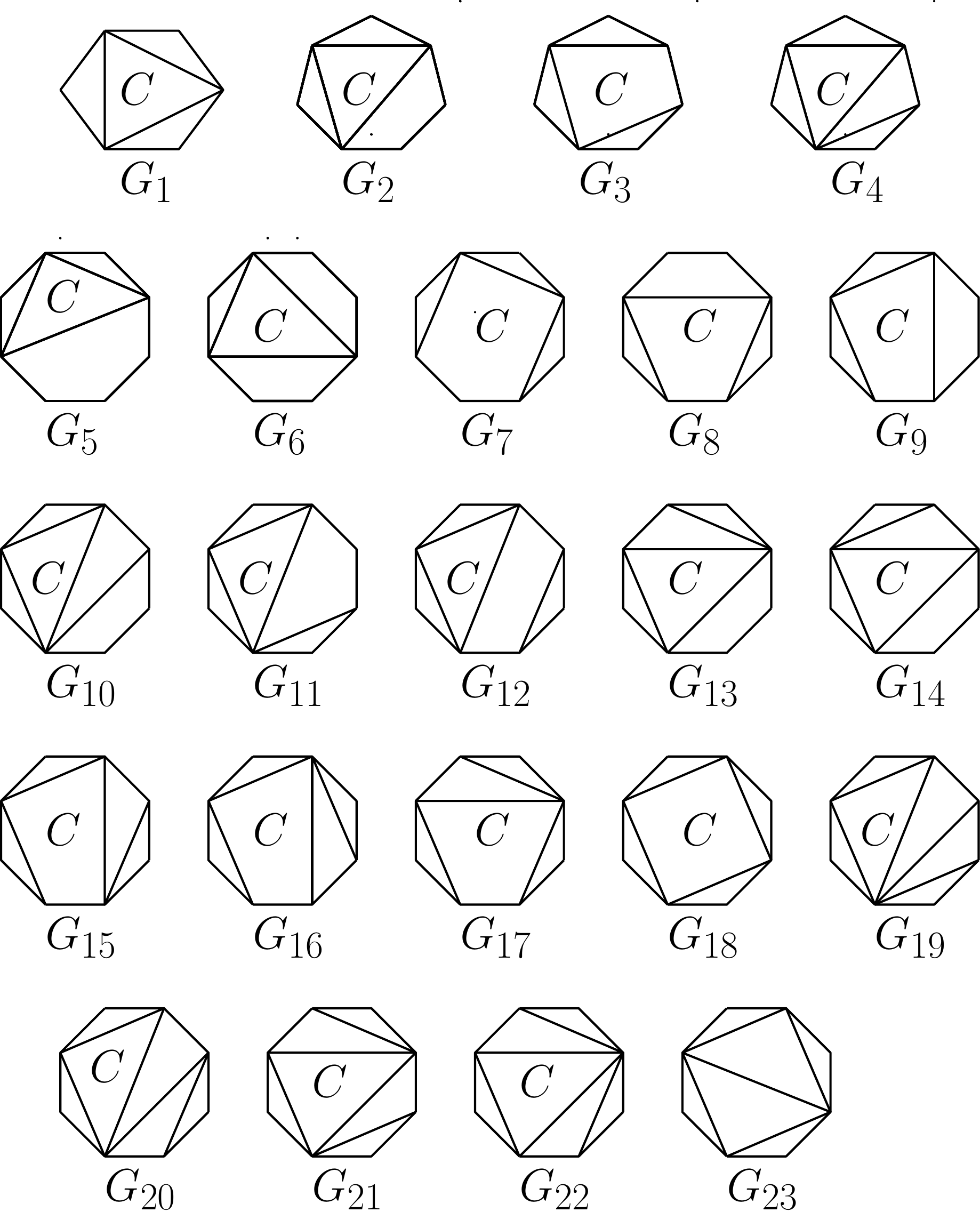}}
 \caption{The 2-connected outerplanar graphs with at most $8$ vertices that are not polygon chains.}
\end{figure}

$$\aligned
  &S(G_1)=\mathbb{Z}_3\oplus \mathbb{Z}_{18}; \ \ S(G_2)=\mathbb{Z}_{75}; \ \ S(G_3)=\mathbb{Z}_3\oplus \mathbb{Z}_{27};\\
  &S(G_4)=\mathbb{Z}_{141};\ \   S(G_5)=\mathbb{Z}_{96}; \ \ S(G_6)=\mathbb{Z}_{104};\\ &S(G_7)=\mathbb{Z}_3\oplus \mathbb{Z}_{36};\ \ S(G_8)=S(G_9)=\mathbb{Z}_{111};\ \ S(G_{10})=\mathbb{Z}_{195};\\
  &S(G_{11})=S(G_{12})=\mathbb{Z}_{204}; \ \ S(G_{13})=S(G_{14})=\mathbb{Z}_{196};\\
  &S(G_{15})=S(G_{16})=S(G_{17})=\mathbb{Z}_{213};\ \
S(G_{18})=\mathbb{Z}_3^2\oplus \mathbb{Z}_{24};\\ &S(G_{19})=S(G_{20})=\mathbb{Z}_3\oplus \mathbb{Z}_{123}; \ \ S(G_{21})=S(G_{22})=\mathbb{Z}_{368};\\
&S(G_{23})=\mathbb{Z}_3\oplus\mathbb{Z}_{120}.
\endaligned
$$

\subsection*{Example 2. Sandpile groups of regular polygon flowers}

A polygon chain is called $r$-\emph{regular} if each polygon in it is the $r$-cycle $C_r$. An $r$-regular chain with $n$ polygons will be denoted by $P_r^n$. Any polygon flower with regular chains as petals is said to be an $r$-\emph{regular polygon flower} (see Figure 4 for examples). By Lemma 3.1, $\tau(P_r^n)$ satisfy the recurrence relation:
$$\tau(P_r^n)=r\tau(P_r^{n-1})-\tau(P_r^{n-2}),\  \tau(P_r^{0})=1,\  \tau(P_r^{1})=r.$$
Hence, $\tau(P_2^n)=n$, and
$$\tau(P_r^n) = \frac{1}{2^{n+1}\sqrt{r^2-4}} \left((r+\sqrt{r^2-4})^{n+1}-(r-\sqrt{r^2-4})^{n+1}\right),\  r>2.$$
Furthermore, let $e$ be a free edge of $P_r^n$. Then we have $\tau(P_2^n\contract e)=1$, while for $r>2$
$$\tau(P_r^n\contract e) = \frac{1}{2^{n+1}\sqrt{r^2-4}} \left((r-2+\sqrt{r^2-4})(r+\sqrt{r^2-4})^{n}-(r-2-\sqrt{r^2-4})(r-\sqrt{r^2-4})^{n}\right).$$

Thus, we can easily deduce the following results:

\smallskip

(i) For the polygon flower $F=F(C_t; P_2^{n_1},\ldots,P_2^{n_t})$ (which is called the \emph{\emph{thick cycle}} in \cite{Alar2017The}), we have
$$S(F)=\mathbb{Z}_{d_1}\oplus \mathbb{Z}_{\frac{d_2}{d_1}}\oplus\cdots\oplus \mathbb{Z}_{\frac{\tau(F)}{d_{t-2}}},$$
where $d_k=gcd(n_{i_1}\cdots
n_{i_k}\ |\ 1\leq i_1<\cdots<i_k\leq t)$ for $k=1,\ldots,t-2$ and $ \tau(F)=\sum_{i=1}^t\prod_{j\neq i}n_j$.

\medskip
(ii) Let $F=F(t, s; r, n)$ denote a polygon flower with center $C_t$ and $s$ petals, where each petal is the polygon chain $P_r^n$. Then
$$S(F)=\mathbb{Z}_{\tau(P_r^n)}^{s-2}\oplus \mathbb{Z}_{(s\tau(P_r^n\contract e)+(t-s)\tau(P_r^n))\tau(P_r^n)}.$$
In particular, for the sun flowers $F=F(t, t; r, n)$, $n=1,2,3$, Corollary 4.5 yields
$$
S(F(t,t;r,1))=\mathbb{Z}_r^{t-2}\oplus \mathbb{Z}_{r(r-1)t};$$
$$
S(F(t,t;r,2))=\mathbb{Z}_{r^2-1}^{t-2}\oplus \mathbb{Z}_{(r^2-1)(r^2-r-1)t};$$
$$
S(F(t,t;r,3))=\mathbb{Z}_{r^3-2r}^{t-2}\oplus \mathbb{Z}_{(r^3-2r)(r^3-r^2-2r+1)t}.$$

\section{The sandpile group of a polygon flower continued}

In this section, we extend the study of the sandpile group of a polygon flower $F=F(C_t;P_1,\ldots,P_t)$ in two directions. One is, if $S(F)$ is cyclic, then we consider whether there exists an edge $e\in E(F)$ which is a generator of $S(F)$. The other direction is to find ways to reduce the relation matrix $R$ further if it is possible. We first give two lemmas.

\begin{lemma} \cite{Brandfonbrener2017Two}
Let $G$ be a graph. Then $e\in E(G)$ is a generator of the sandpile group $S(G)$ if and only if
$$gcd(\tau(G),\tau(G\contract e))=1.$$
\end{lemma}

\begin{lemma}
Let $G_n=G_n(k_1,\ldots,k_n)\ (k_i\geq 2)$ be a polygon chain, and $e_0,\ldots,e_n$ be as defined in Section 3. Then
$$\tau(G_n\contract e_0)\tau(G_n\contract e_n)-\tau(G_n)\tau(G_n\contract e_0\contract e_n)=1.$$
\end{lemma}

\begin{proof} We shall prove the result by induction on $n$. For $n=1$, we have $$
 \tau(G_1)=k_1;\ \tau(G_1\contract e_0\contract e_1)=k_1-2;\ \tau(G_1\contract e_0)=\tau(G_1\contract e_1)=k_1-1.$$ So the result holds for $i=1$.

 For the induction step, we use (3.2) to obtain
 $$\aligned
 &\tau(G_{n})=(k_n-1)\tau(G_{n-1})+\tau(G_{n-1}\contract e_{n-1});\\ &\tau(G_n\contract e_0\contract e_n)=(k_n-2)\tau(G_{n-1}\contract e_0)+\tau(G_{n-1}\contract e_0\contract e_{n-1});\\
 &\tau(G_{n}\contract e_0)=(k_n-1)\tau(G_{n-1}\contract e_0)+\tau(G_{n-1}\contract e_0\contract e_{n-1}); \\&\tau(G_n\contract e_n)=(k_n-2)\tau(G_{n-1})+\tau(G_{n-1}\contract e_{n-1}).
 \endaligned
 $$
 This gives
 $$\aligned \tau(G_{n}\contract e_0)&\tau(G_n\contract e_n) - \tau(G_{n})\tau(G_n\contract e_0\contract e_n)\\
 & = \tau(G_{n-1}\contract e_0)\tau(G_{n-1}\contract e_{n-1})-\tau(G_{n-1})\tau(G_{n-1}\contract e_0\contract e_{n-1}),\endaligned$$
 which is equal to $1$ by the induction hypothesis.
\end{proof}

 We are ready to give the first main result of this section.

\begin{thm}\label{thm:5.3}
Let $F=F(C_t;P_1,\ldots,P_t)$ be a polygon flower, and let $f_1,\ldots, f_t$ be the edges chosen in Theorem 4.1. Then $f_i\  (1\leq i\leq t)$ is a generator of $S(F)$ if and only if
$$ m(\tau(P_1),\ldots,\tau(P_{i-1}), \tau(P_{i+1}),\ldots,\tau(P_t))=1.$$
\end{thm}

\begin{proof}
By symmetry, we may assume that $i=1$. By Lemma 5.1, $f_1$ is a generator of $S(F)$ if and only if
$$ gcd(\tau(F), \tau(F\contract f_1))=1.$$
 Let $C_t=e_1\ldots e_t$. First by formula (4.1), we have
$$\tau(F)=\tau(P_1\contract e_1)\prod_{j=2}^t\tau(P_j)+\tau(P_1)\sum_{i=2}^t\tau(P_i\contract e_i)\prod_{j\neq i;j\geq 2}\tau(P_j);$$
$$\tau(F\contract f_1)=\tau(P_1\contract e_1\contract f_1)\prod_{j=2}^t\tau(P_j)+\tau(P_1\contract f_1)\sum_{i=2}^t\tau(P_i\contract e_i)\prod_{j\neq i;j\geq 2}\tau(P_j).$$
Then solving these equations for $\tau=\prod_{j=2}^t\tau(P_j)$ and $\tau'=\sum_{i=2}^t\tau(P_i\contract e_i)\prod_{j\neq i;j\geq 2}\tau(P_j)$ by using Lemma 5.2, we obtain
$$\tau=\tau(P_1\contract f_1)\tau(F)-\tau(P_1)\tau(F\contract f_1);$$
$$\tau'=-\tau(P_1\contract e_1\contract f_1)\tau(F)+\tau(P_1\contract e_1)\tau(F\contract f_1).$$
Hence
$$gcd(\tau(F), \tau(F\contract f_1))=gcd(\tau,\tau').$$
It remains to show that
$$gcd(\tau,\tau')=1\Leftrightarrow m(\tau(P_2),\ldots,\tau(P_t))=1.$$
First, if $m(\tau(P_2),\ldots,\tau(P_t))=1$, then for any prime $q$, $q$ divides at most one of $\tau(P_2),\ldots,\tau(P_t)$, say $\tau(P_2)$, then $q$ divides every term in the sum defining $\tau'$ except possible $\tau(P_2\contract e_2)\prod_{j=3}^t\tau(P_j)$. Moreover it divides this term if and only if it divides $\tau(P_2\contract e_2)$. Recall the fact that $gcd(\tau(P_2),\tau(P_2\contract e_2))=1$. This shows that $q$ does not divide $\tau'$, and hence $gcd(\tau,\tau')=1.$ On the other hand, if $m(\tau(P_2),\ldots,\tau(P_t))>1$, then there exists a prime $q$ such that $q$ divides at least two of $\tau(P_2),\ldots,\tau(P_t)$. Clearly, then $q$ divides $\tau$ and $\tau'$. This completes the proof.
\end{proof}

Now we can give a complete answer to the question whether there exists a generating edge in a polygon flower. Note that this is only a sufficient condition for the sandpile group being cyclic.

\begin{thm}\label{thm:5.4}
Let $F=F(C_t;P_1,\ldots,P_t)$ be a polygon flower. Then there exists an edge  $e\in E(F)$ that generates $S(F)$ if and only if
there exists at least one $i$ such that
$$ m(\tau(P_1),\ldots,\tau(P_{i-1}), \tau(P_{i+1}),\ldots,\tau(P_t))=1.\eqno (5.1)$$
\end{thm}

\begin{proof}
Let $f_1,\ldots, f_t$ be the edges chosen in Theorem 4.1. Suppose (5.1) holds. Then by Theorem 5.3, there is an $i$ such that $f_i$ is a generator of $S(F)$. Conversely, if (5.1) does not hold, then by Theorem 5.3, none of $f_1,\ldots, f_t$  generates $S(F)$. Any other edge $e\in E(F)$, if $e\in P_i$, then $e$ can be expressed in $S(F)$ as a multiple of $f_i$, say $e=af_i$, where the coefficient $a$ can be determined by Theorem 3.2.
So $e$ cannot be a generator of $S(F)$ since $f_i$ is not. Thus we complete the proof.
\end{proof}

Moreover, we see from the above that, for any $e\in E(F)$, if $e=af_i$ for some $i$, then $e$ is a generator of $S(F)$ if and only if
$$ m(\tau(P_1),\ldots,\tau(P_{i-1}), \tau(P_{i+1}),\ldots,\tau(P_t))=1 \ \mbox{and}\  gcd(a, \tau(F))=1.$$

Now we turn to the second direction. Recall that, for a polygon flower $F=F(C_t;P_1,\ldots,P_t)$, the minimum number of generators is $\mu(F)=m(\tau(P_1),\ldots,\tau(P_t))-1$ (\ or $1$ if this quantity is zero). So if $\tau(P_{i_1}),\ldots,\tau(P_{i_k})$ are pairwise relatively prime, they contribute at most one element to a minimum set of generators. Motivated by this, we introduce the notion of a prime partition.

Given an integral vector $\vec{a}=(a_1, \ldots, a_t)$,
let $A_1\cup \cdots\cup A_k$ be a partition of $\{1,\ldots,t\}$. We define $\alpha_i=: \prod_{j\in A_i}a_j,\  i=1,\ldots,k$. A partition of $\{1,\ldots,t\}$ is called a \emph{prime partition} of $\vec{a}=(a_1, \ldots, a_t)$ if it satisfies the following two properties:

(1) $gcd(a_i,a_j)=1$ if $i$ and $j$ belong to the same part $A_l$ of the partition;

(2) $gcd (\alpha_i,\alpha_j)\neq 1$ for any $1\leq i<j\leq k$.

Let us consider the two examples we gave before: $\vec{a}=(2, 2, 3,3, 5,5)$ and $\vec{b}=(6, 10, 15, 105)$. It is easy to see that $\{1,3\}\cup\{2,5\}\cup\{4,6\}$ and $\{1,3,5\}\cup\{2,4,6\}$ are two prime partitions of $\vec{a}$. On the other hand, $\vec{b}$ has only one (trivial) prime partition $\{1\}\cup\{2\}\cup\{3\}\cup\{4\}$. Note that by the property (1), the number of parts in any prime partition of $\vec{a}=(a_1, \ldots, a_t)$ is at least $m(\vec{a})$.

In the following, we shall show that we can reduce the relation matrix $R$ further by using any non-trivial prime partition of $\vec{p}=(\tau(P_1),\ldots,\tau(P_t))$.

\begin{lemma}
Let $\vec{a}=(a_1, \ldots, a_t)$ be a positive integral vector, and $c_1,\ldots,c_t$ be integers. Suppose that $$A_1\cup A_2\cup\cdots\cup A_k=\{1,\cdots,i_1\}\cup \{i_1+1,\ldots,i_2\}\cup\cdots\cup\{i_{k-1}+1,\ldots,t\}$$ is a prime partition of $\vec{a}$. Then we have
$$M=\left(
\begin{array}{cccccc}
  a_1&-a_2&0&\cdots&0&0\\
  0&a_2&-a_3&\cdots&0&0\\
  \cdots&\cdots&\cdots&\ddots&\cdots&\cdots\\
  0&0&0&\cdots&a_{t-1}&-a_t\\
  c_1&c_2&c_3&\cdots&c_{t-1}&c_t
 \end{array}
\right)\sim $$
$$\left(
\begin{array}{cccccccccc}
  1&0&0&\cdots&0&0&0&\cdots&0&0\\
  0&1&0&\cdots&0&0&0&\cdots&0&0\\
  \cdots&\cdots&\cdots&\ddots&\cdots&\cdots&\cdots&\ddots&\cdots&\cdots\\
  0&0&0&\cdots&1&0&0&\cdots&0&0\\
  0&0&0&\cdots&0&\alpha_1&-\alpha_2&\cdots&0&0\\
  \cdots&\cdots&\cdots&\ddots&\cdots&\cdots&\cdots&\ddots&\cdots&\cdots\\
  0&0&0&\cdots&0&0&0&\cdots&\alpha_{k-1}&-\alpha_k\\
  0&0&0&\cdots&0&c_1'&c_2'&\cdots&c_{k-1}'&c_k'
 \end{array}
\right),$$
where $\alpha_i=\prod_{j\in A_i}a_j$, and $c_i'=\prod_{j\in A_i}a_j\sum_{l\in A_i}\frac{c_l}{a_l}, \ i=1,\ldots,k.$
\end{lemma}

\begin{proof}
We shall show that there exist invertible matrices $X$ and $Y$ such that $XMY=N$, where $N$ is the second matrix shown in the lemma. First, we suppose that $k=1$, that is, $a_1, \ldots, a_t$ are pairwise relatively prime. Since $gcd(a_1,a_2)=1$, there exist integers $y_1,y_2$ such that $y_1 a_1+y_2 a_2=1$. Let
$$Y_1 = diag\left(\left(
\begin{array}{cc}
y_1&a_2\\
-y_2&a_1
\end{array}
\right), I_{t-2}\right).$$
Clearly, $Y_1$ is invertible since $det(Y_1)=1$ and
$$MY_1 = \left(
\begin{array}{cccccc}
  1&0&0&\cdots&0&0\\
  -y_2 a_2&a_1a_2&-a_3&\cdots&0&0\\
  \cdots&\cdots&\cdots&\ddots&\cdots&\cdots\\
  0&0&0&\cdots&a_{t-1}&-a_t\\
  y_1 c_1-y_2 c_2&c_1a_2+c_2a_1&c_3&\cdots&c_{t-1}&c_t
\end{array}
\right).$$
Similarly, since $gcd(a_1a_2,a_3)=1$, there exist integers $y_2',y_3$ such that $y_2'a_1a_2+y_3 a_3=1$. Let
$$Y_2 = diag\left(I_1,\left(
\begin{array}{cc}
y_2'&a_3\\
-y_3&a_1a_2
\end{array}
\right), I_{t-3}\right).$$
Then we have
$$ MY_1Y_2=\left(
\begin{array}{ccccccc}
  1&0&0&0&\cdots&0&0\\
  *&1&0&0&\cdots&0&0\\
  0&*&a_1a_2a_3&-a_4&\cdots&0&0\\
  \cdots&\cdots&\cdots&\cdots&\ddots&\cdots&\cdots\\
  0&0&0&0&\cdots&a_{t-1}&-a_t\\
  *&*&c_1a_2a_3+c_2a_1a_3+ c_3a_1a_2&c_4&\cdots&c_{t-1}&c_t
\end{array}
\right).$$
And so on,
$$MY_1\cdots Y_{t-1} =\left(
\begin{array}{cccccc}
   1&0&0&\cdots&0&0\\
  *&1&0&\cdots&0&0\\
  \cdots&\cdots&\cdots&\ddots&\cdots&\cdots\\
  0&0&0&\cdots&1&0\\
  *&*&*&\cdots&*&\prod_{j=1}^{t}a_j\sum_{l=1}^t \frac{c_l}{a_l}
\end{array}
\right),$$
where
$$Y_i = diag\left(I_{i-1},\left(
\begin{array}{cc}
 y_i' & a_{i+1}\\
-y_{i+1} & \prod_{j=1}^ia_j
\end{array}
\right), I_{t-i-1}\right).$$

By setting $Y=Y_1\cdots Y_{t-1}$, it is clear that there exists $X$ such that $$XMY =\left(
\begin{array}{cccccc}
   1&0&0&\cdots&0&0\\
  0&1&0&\cdots&0&0\\
  \cdots&\cdots&\cdots&\ddots&\cdots&\cdots\\
  0&0&0&\cdots&1&0\\
  0&0&0&\cdots&0&\prod_{j=1}^{t}a_j\sum_{l=1}^t \frac{c_l}{a_l}
 \end{array}
\right).$$
Notice that $MY_1\cdots Y_{t-1}$ is obtained from $M$ by performing column operations on $M$ step by step. In the $i$-th step, only columns $i$ and $i+1$ are changed.

For $k\geq 2$, we do the similar column operations for each block separately. It suffices to consider $k=2$. Suppose the partition is $\{1,\ldots,s\}\cup\{s+1,\ldots,t\}$. Let $Y,Y'$ be the corresponding matrices of column operations, respectively. Then $MYY' $ is equal to the matrix

$$\left(
\begin{array}{cccccccccc}
  1&0&0&\cdots&0&0&0&\cdots &0&0\\
  *&1&0&\cdots&0&0&0&\cdots &0&0\\
  \cdots&\cdots&\cdots&\ddots&\cdots&\cdots&\cdots&\ddots &\cdots&\cdots\\
  0&0&0&\cdots&1&0&0&\cdots &0&0\\
  0&0&0&\cdots&0&\prod_{j=1}^sa_j&*&\cdots &*&-\prod_{j=s+1}^ta_j\\
  0&0&0&\cdots&0&0&1&\cdots &*&0\\
  \cdots&\cdots&\cdots&\ddots&\cdots&\cdots&\cdots&\ddots &\cdots&\cdots\\
  0&0&0&\cdots&0&0&*&\cdots &1&0\\
  *&*&*&\cdots&*&\prod_{j=1}^sa_j\sum_{l=1}^s \frac{c_l}{a_l}&*&\cdots&*&\prod_{j=s+1}^ta_j\sum_{l=s+1}^t \frac{c_l}{a_l}
 \end{array}
\right).$$
It follows that
$$M\sim \left(
\begin{array}{ccccccc}
  1&0&0&\cdots&0&0&0\\
  0&1&0&\cdots&0&0&0\\
  \cdots&\cdots&\cdots&\ddots&\cdots&\cdots&\cdots\\
  0&0&0&\cdots&1&0&0\\
  0&0&0&\cdots&0&\prod_{j=1}^sa_j&-\prod_{j=s+1}^ta_j\\
  0&0&0&\cdots&0&\prod_{j=1}^sa_j\sum_{l=1}^s \frac{c_l}{a_l}&\prod_{j=s+1}^ta_j\sum_{l=s+1}^t \frac{c_l}{a_l}
 \end{array}
\right).$$
 This completes the proof.
\end{proof}

From the above lemma, we immediately derive the following result.

\begin{thm}\label{thm:5.6}
Let $F=F(C_t;P_1,\ldots,P_t)$ be a polygon flower, and let $A_1\cup\cdots\cup A_k$ be a prime partition of $\vec{p}=(\tau(P_1),\ldots,\tau(P_t))$. Let $$\alpha_i=\prod_{j\in A_i}\tau(P_j),\ \ \beta_i=\alpha_i\sum_{l\in A_i}\frac{\tau(P_l\contract e_l)}{\tau(P_l)}, \ i=1,\ldots,k.$$ Then
the relation matrix $R$ is equivalent to
 $ diag(I_{t-k}, R'),$
where $$R'=\left(
\begin{array}{cccccc}
  \alpha_1&-\alpha_2&0&\cdots&0&0\\
  0&\alpha_2&-\alpha_3&\cdots&0&0\\
  \cdots&\cdots&\cdots&\ddots&\cdots&\cdots\\
  0&0&0&\cdots&\alpha_{k-1}&-\alpha_k\\
  \beta_1&\beta_2&\beta_3&\cdots&\beta_{k-1}&\beta_k
 \end{array}
\right)$$
and $$gcd(\alpha_i,\beta_i)=1,\ i=1,\ldots,k.$$
\end{thm}

\begin{proof}
The first part follows directly from Lemma 5.5. For the second part, without loss of generality, let $A_1=\{1,\dots,r\}$. Then $\alpha_1=\prod_{i=1}^r\tau(P_i)$ and $$\beta_1=\tau(P_1\contract e_1)\prod_{i=2}^r\tau(P_i)+\tau(P_1)\tau(P_2\contract e_2)\prod_{i=3}^r\tau(P_i)+\cdots+\prod_{i=1}^{r-1}\tau(P_i)\tau(P_r\contract e_r).$$ If $q>1$ is a prime dividing $gcd(\alpha_1,\beta_1)$, then there exists a unique $i\in A_1$ such that $q\mid\tau(P_i)$ since $gcd(\tau(P_i),\tau(P_j))=1$ for any pair $i,j$ in $A_1$. Say $q\mid\tau(P_1)$. Combining this with the fact that $q\mid\beta_1$, we have $q\mid\tau(P_1\contract e_1)\prod_{i=2}^r\tau(P_i)$, so $q\mid\tau(P_1\contract e_1)$. This contradicts the fact that $gcd(\tau(P_1), \tau(P_1\contract e_1))=1$. Hence $gcd(\alpha_1,\beta_1)=1$. This completes the proof.
\end{proof}

Note that $R$ and $R'$ not only have the same number of non-trivial invariant factors, but have the same form and some other properties. Similarly as in the proof of Theorem 4.3, we have the following corollary.

\begin{cor}
Let $F=F(C_t;P_1,\ldots,P_t)$ be a polygon flower, and let $A_1\cup\cdots\cup A_k$ be a prime partition of $\vec{p}=(\tau(P_1),\ldots,\tau(P_t))$. Let $\alpha_i=\prod_{j\in A_i}\tau(P_j)$, $i=1,\ldots,k$. Then
$$ S(F)=\mathbb{Z}_{\frac{d_1'}{d_0'}}\oplus \mathbb{Z}_{\frac{d_2'}{d_1'}}\oplus\cdots\oplus \mathbb{Z}_{\frac{d_{k-2}'}{d_{k-3}'
}}\oplus \mathbb{Z}_{\frac{\tau(F)}{d_{k-2}'}},\eqno (5.2)$$
where $d_0'=1$ and $d_j'=gcd(\alpha_{i_1}\cdots\alpha_{i_j}\ |\ 1\leq i_1<\cdots<i_j\leq k)$ for $j=1,\ldots,k-2$.
And the minimum number of generators of $S(F)$ is
$$\mu(F) = k-1-k_0', \eqno (5.3)$$
where $k_0' = \max\{i \mid d_i'=1,\  i=0,\ldots,k-2\}$.
\end{cor}

The above result shows that for a given polygon flower $F=F(C_t;P_1,\ldots,P_t)$, we get the complete information about $S(F)$ by doing the following.

Step 1: Compute $\tau(P_1),\ldots,\tau(P_t)$ and find a prime partition of $\vec{p}=(\tau(P_1),\ldots,\tau(P_t))$ such that the number $k$ of parts is as small as possible.

Step 2: Compute invariant factors $d_i',\  i=1,\ldots,k-2$ and $\tau(F)$.

Conversely, we can use the above result to construct polygon flowers for which the minimum number of generators is equal to any given positive integer. In particular, to construct graphs whose sandpile group is cyclic. For example, given any two families of pairwise co-prime integers
$k_1,\ldots,k_s$ and $k_{s+1},\ldots,k_t$, the sandpile group of $F=F(C_t;C_{k_1},\ldots,C_{k_t})$ is cyclic since $m(k_1,\ldots,k_t)$ is at most $2$.

\section{Concluding remarks}

In this paper, we study the sandpile group of polygon flowers. Starting from the natural set of generators $\delta_e$, we first use the structure properties of the graph to reduce the number of generators to be as small as possible, and find a relation matrix of these generators. Then by analyzing the relation matrix, we give an explicit formula for the sandpile group of a polygon flower. Note that polygon flowers are outerplanar graphs. We think that this method can be used to study the sandpile groups of general outerplanar graphs, and more generally of any graphs whose tree-width is at most two.

Let $G$ be a bi-connected outerplanar graph, and let $G^*$ denote the inner dual (deleting the vertex
corresponding to the unbounded face from the usual dual) of $G$. Then it is well known that $G^*$ is a tree. Let $l(T)$ denote the number of leaves in tree $T$. It is easy to see that for a polygon flower $F=F(C_t;P_1,\ldots,P_t)$ with $s$ petals, $l(F^*)=s$. So Theorem 4.4 tells us that the minimum number of generators of $S(F)$ is at most $l(F^*)-1$. In fact, using the method of this paper, it is not difficult to generalize this result to any bi-connected outerplanar graph $G$ to obtain that $\mu(G)\leq l(G^*)-1$.

Also note that a maximal induced path in $G^*$ corresponds to a polygon chain in $G$, so every bi-connected outerplanar graph $G$ has a unique polygon chain decomposition corresponding to the maximal path decomposition of $G^*$. Theorem 4.4 says that for a polygon flower $F$, $\mu(F)$ is completely determined by the numbers of spanning trees of these polygon chains. We conjecture that the same holds for any outerplanar graph.

\begin{conjecture}
Let $G$ be a bi-connected outerplanar graph, and $P_1\cup \cdots\cup P_t$ be the polygon chain decomposition defined as above, then $\mu(G)$ is determined by the numbers $\tau(P_1),\ldots,\tau(P_t)$.
\end{conjecture}

\bibliographystyle{abbrv}

\bibliography{polygonflower1}

\end{document}